\documentclass[12pt]{amsart} 

\usepackage[margin=2.9cm]{geometry}             
\usepackage{color,ulem}
\usepackage{graphicx}
\usepackage{amssymb, enumerate,bm}
\usepackage[matrix,arrow,ps]{xy}
\usepackage{epstopdf, amsmath} 
\usepackage{paralist}
\newcommand{\C}{\mathbb C}

\newcommand{\R}{\mathbb R}

\newcommand{\transp}{\,^t}

\newcommand{\aaa}{\mathfrak{(a)}}

\newcommand{\bbb}{\mathfrak{(b)}}

\newtheorem{theo}{Theorem}[section]
\newtheorem{lemma}[theo]{Lemma}
\newtheorem{cor}[theo]{Corollary}
\newtheorem{prop}[theo]{Proposition}

\theoremstyle{remark}
\newtheorem{remark}[theo]{Remark}
\theoremstyle{example}
\newtheorem{example}[theo]{Example}
\theoremstyle{definition}
\newtheorem{defi}[theo]{Definition}
\numberwithin{equation}{section}

\begin{document}

\begin{abstract}
We investigate some of the properties of the Levi map and related objects for nondegenerate CR submanifolds in $\C^N$ of codimension $d$.    
\end{abstract} 
\thanks{Research of the first author 
author was support by a URB  grant from the American University of Beirut,  and by the Center for Advanced Mathematical Sciences.}

\author[Bertrand, Meylan]{Florian Bertrand and Francine Meylan}
\title[On the Levi map of nondegenerate CR submanifolds]{On the Levi map of nondegenerate CR submanifolds}

\subjclass[2010]{}

\keywords{}
\thanks{}

\maketitle


\section*{Introduction}

Let $M \subset \C^N$ be a CR submanifold of codimension $d$. In order to understand the geometry of M and properties of its maps, it is important to associate invariant objects. When $M$ is nondegenerate in a suitable sense, one of these natural objects is the Levi map. Geometric properties of the Levi map have strong consequences on holomorphic extension and regularity properties of CR mappings for $M$. For instance, when the convex hull of the range of the Levi map, that is, the Levi cone, has a nonempty interior, then there is wedge extendability for CR functions of $M$ \cite{bo-po}
(see also p. 200-202 \cite{bo}).  Inspired by the work of Boggess \cite{bo}, we aim  to study and characterize the geometry of the Levi map. In particular, one of our motivations is to understand  the link between the Levi map and the different notions of nondegeneracy that have led to the 
2-jet determination of CR automorphisms of $M$ (see for instance \cite{be1, be-bl-me,be-me,ch-mo,la-mi,tu,tu3,za}). Indeed, in addition to the Levi 
cone, the range of the Levi map, referred to as its Levi range, seems to play a natural role. For instance, we show in Theorem \ref{theorange} 
and its Corollary \ref{cordisc} that the Levi range of $M$ has nonempty interior if and only if  there exists a {\it nondefective} analytic disc of a 
special form. The notion of {\it defect},  which is crucial in the wedge extendability of CR functions,  was introduced by Tumanov \cite{tu0}, and 
has recently appeared to be important in the study of jet determination problems \cite{be-me, tu3}. Furthermore, we show in Proposition \ref{propd=2} 
and Theorem \ref{theod=3} that in codimension $2$ and $3$, the Levi range of $M$ always has  nonempty interior, provided that the Levi cone 
has   nonempty interior. We also study manifolds for which the Levi cone and the Levi range coincide. In particular, we prove in Theorem
\ref{theorangecone} that they always coincide for weakly pseudoconvex manifolds of codimension 2 and for any submanifold of codimension $2$ in $\C^4$.

\section{Preliminaries}
\subsection{Nondegenerate generic submanifolds}

We consider  a smooth generic  real submanifold  $M \subset \C^{N}$ of real codimension $d\ge 1$ of the form 
\begin{equation}\label{eqred}
\begin{cases}
\Re e  w_1= \transp\bar z A_1 z+ \ldots\\
\ \ \ \ \vdots \\
\Re e  w_d = \transp\bar z A_d z+ \ldots
\end{cases}
\end{equation}
where $A_1,\hdots,A_d$ are $n\times n$ Hermitian matrices and where the dots represent terms in $z$ and $\Im m w$ of higher (weighted) order (see \cite{bo,bl-me}). In case there are no dots in \eqref{eqred}, we refer to $M$ as a model quadric and we say that $M$ is defined by the matrices $A_1,\ldots,A_d$.  We now recall a series of different notions of nondegeneracy which have appeared in relation with jet determination of CR diffeomorphisms.

\begin{defi} [\cite{be1}]\label{defnondegbe}
The submanifold $M$ given by \eqref{eqred} is  {\it Levi nondegenerate} at 0 if the following two conditions are both satisfied
\begin{center}
\begin{tabular}{cl}
$\aaa$ & $A_1$,...,$A_d$ are linearly independent (equivalently on $\R$ or $\C$),\\
\\
$\bbb$ & $\bigcap_{j=1}^d\mathrm{Ker}A_j=\{0\}$.
\end{tabular}
\end{center}
\end{defi}

Note that Condition $\aaa$ implies  $d \le n^2$  and  that $M$ is  of finite type, while Condition $\bbb$ implies that $M$ is holomorphically nondegenerate. Moreover,  when $M$ is a hypersurface, Condition $\bbb$ implies Condition  $\aaa$. 

\begin{defi} [\cite{tu}] \label{defnondegtu}
The submanifold $M$ is  {\it strongly Levi nondegenerate} (resp. {\it strongly pseudoconvex}) at 0 if there exists a real linear combination of the $A_j's$ which is invertible (resp. positive definite).
\end{defi}

We point out that in Section \ref{sec3} we will work with the weak version of pseudoconvexity (see Definition  \ref{defweak}).
Note that  if $M$ is  strongly Levi nondegenerate then it satisfies Condition   $\bbb.$ 

\begin{remark}
Any submanifold $M$ of maximal codimension $d=n^2$   satisfying Condition $\aaa$  is strongly pseudoconvex.  
Indeed, after a  holomorphic linear change of coordinates, $M$ can be written as	
\begin{equation*}
\Re e  w_{j,k}=\transp \overline{z} A_{j,k} z \ \ 1\leq j,k \leq n
\end{equation*}
where $A_{j,k}, 1\leq j,k \leq n$ is the canonical basis of the space of Hermitian $n\times n$ matrices. In particular, $\Re e  w_{j,j}=|z_j|^2$, which implies  that $M$  is strongly  pseudoconvex.
\end{remark}

We  now recall  the following notion  which implies conditions $\aaa$ and $\bbb$, and the strong  Levi nondegeneracy.

\begin{defi} [\cite{be-bl-me}]\label{deffully} 
Assume that $M$ is  strongly Levi 
nondegenerate. We denote by $A:=\sum c_jA_j$ an invertible real linear combination of the $A_j$'s. 
We say that  $M$ is {\it fully nondegenerate} at 0 if  there exists $V \in \C^{n}$ 
such that both $$\transp \overline D A^{-1}D \ \ \ \mbox{ and } \ \ \ \Re e \left(\transp \overline D A^{-1}D\right)$$ are invertible, where $D$ is the $n \times d$ matrix whose $j^{th}$ column is $A_jV$.  
\end{defi}
Note that this imposes the strong restriction $d\leq n$. When $M$ is fully nondegenerate then the matrix $D$ is of (complex) rank $d$, that is, the vectors $A_1V,\ldots,A_dV$ are $\C$-linearly independent. It is in fact equivalent when $M$ is strongly pseudoconvex since in that case $A$ can be assumed to be the identity matrix. 
Also, in the hypersurface case, the fully nondegeneracy condition is  equivalent to Condition $\bbb.$ A more flexible notion, which  allows to work with $d \le 2n$  instead, is the following.     
\begin{defi} [\cite{be-me}]\label{defect} 
The submanifold $M$ given by  (\ref{eqred}) is {\it$\mathfrak{D}$-nondegenerate} at 0 if $M$ is strongly Levi 
nondegenerate and  if there exists $V \in \C^n$  such that  $\Re e (\transp \overline D A^{-1}D)$ is invertible. Here, $A$ and $D$ still denotes the matrices introduced in Definition \ref{deffully}. 
\end{defi}

It  follows that when $M$ is {\it$\mathfrak{D}$-nondegenerate},   there exists $V \in \C^n$  such that the vectors $A_1V,\ldots,A_dV$ are $\R$-linearly independent in $\C^n$, leading to the restriction $d\leq 2n$. We emphasize that under the assumption that $M$ is {\it$\mathfrak{D}$-nondegenerate}, sufficiently smooth $CR$ germs automorphisms at $0$ of $M$ are uniquely determined by their $2-$jet \cite{be-me}.

\subsection{The Levi map and associated objects}

Let $M \subset \C^{N}$ be a real submanifold. We respectively denote by $TM$, $NM$ and $N^*M$, its tangent, normal and conormal bundles. We also denote by $\pi: T{\C^N} \to NM$ the orthogonal projection onto $NM$. Following \cite{ber, bo} (see also \cite{bl-me}), we define 
\begin{defi}
The {\it Levi map} at $p \in M$, $L_p: T_p^{(1,0)}M \to N_pM$, where $T_p^{(1,0)}M$ denotes the space of holomorphic vector fields, is defined by 
$$L_p(V_p):=\frac{1}{2i}\pi_p\left(J_{st}\left[\overline{V},V\right]_p\right)$$
where $V$ is any holomorphic vector field with $V(p)=V_p \in T_p^{(1,0)}M$ and where $J_{st}$ denotes the standard complex structure on $\C^N$.  
The image $L_p\left(T_p^{(1,0)}M\right)$ of the Levi map is referred to as its {\it Levi range at $p$}. The convex  hull of the Levi range at $p$ is called the {\it Levi cone at $p$} and is denoted by $\Gamma_p\subset N_pM$. Finally, the {\it dual Levi cone} $\Gamma^*_p\subset N^*_pM$ is defined by 
$$\Gamma^*_p:=\{ c \in\R^d \ | \ -c\partial \overline{\partial}\rho\left(V_p,\overline{V_p}\right)>0 \  \ \forall \ V_p \ \in T_p^{(1,0)}M\}.$$
\end{defi}
In coordinates in which $M$ is given by \eqref{eqred}, the Levi map at $p=0$ is given by $$L_0: \C^n \to \R^d$$ with
$$L_0(V)=\left(\transp \overline{V}A_1V,\transp \overline{V}A_2V,\ldots,\transp \overline{V}A_dV\right).$$
In such case, the Levi cone is the convex hull 
$$\Gamma_0={\rm conv} \left\{\left(\transp \overline{V}A_1V,\transp \overline{V}A_2V,\ldots,\transp \overline{V}A_dV\right) \ | \ V\in \C^n \right\},$$
while the  dual Levi cone is 
$$\Gamma^*_p=\left\{ c \in\R^d \ | \ \sum_{j=1}^dc_jA_j>0 \right\}.$$  
\begin{remark}
Note that for any $\lambda \in \C$ and $V \in \C^n$ we have $L_0(\lambda V)=|\lambda|^2L_0(V)$. It follows that if a point $Y \in \R^d$ lies in the Levi range at $0$ of $M$, then the semi-line through $Y$, namely $\{ \lambda Y \ | \ \lambda \geq 0\}$,  is contained in the Levi range.
\end{remark}     
We recall the following result due to Tumanov \cite{tu} (see also \cite{bl-me}).

\begin{prop}\label{propcone1}
Let $M$ be a real submanifold given by \eqref{eqred}. 
\begin{enumerate}[i.]
\item The submanifold $M$ is strongly pseudoconvex at $0$ if and only if $\Gamma^*_0\neq \emptyset$.
\item Condition $\aaa$ holds if and only if the interior of the Levi cone $\Gamma_0$ is  nonempty if and only if the Levi range $L_0(\C^n)$ is not included in a hyperplane in $\R^d$. 
\end{enumerate}
\end{prop}

As pointed out by Tumanov \cite{tu,tu2},  if $M \subset \C^N$ is strongly pseudoconvex at $0$ then its Levi cone $\Gamma^*_0$ does not contain a line through the origin in $\R^d$. Indeed, if  $M$ is strongly pseudoconvex at $0$, then one can assume that, after a linear change of coordinate, $A_1=I_n$. The Levi range and thus the Levi 
cone is included in $(\R^+\times \R^{d-1}) \cup \{0\} \subset \R^d$ and so does not contain a line through $0$. 

The converse of this statement is false in general. There are indeed trivial examples of submanifolds which are not strongly pseudoconvex and whose Levi cone does not contain a line through 0; e.g. the submanifold defined by $\Re e w_1=|z_1|^2, \Re e w_2=|z_2|^2$ in $\C^5$ which is  not strongly Levi nondegenerate, or  the one defined by $\Re e w_1=|z_1|^2, 
\Re e w_2=|z_1|^2$ in $\C^4$ which does not satisfy Condition $\aaa$. A more interesting example is the strongly Levi 
nondegenerate submanifold  in Example \ref{exline} which satisfies Condition $\aaa$ and is not strongly pseudoconvex, but whose Levi cone 
does not contain a line through $0$. Note however that its Levi cone contains entire lines which do not pass through the origin. 
The following question then remains open.

\vspace{0.1cm}

\noindent {\bf Question:}  Assume that $M$ (of the form \eqref{eqred}) is a strongly Levi nondegenerate submanifold satisfying Condition $\aaa$ with a Levi cone at $0$ that {\it does not contain an entire line} (not passing through the origin). Is $M$ strongly pseudoconvex?
%
%
%
%
%

%
\subsection{Examples}

We now provide some examples of Levi ranges and cones. The first four examples deal with codimension $d=2$.		
	
\begin{example}\label{excorner}
We consider the strongly pseudoconvex quadric $M\subset \C^4$ of codimension $2$ defined by   
$$A_1=
\begin{pmatrix}
	1& 0 \\ 0 &0 
\end{pmatrix} \ \ \ \ A_2=
\begin{pmatrix}
	0& 0 \\ 0 &1 
\end{pmatrix}.$$
The Levi range of $M$ is then
 $$\left\{\left(|V_1|^2,|V_2|^2\right) \ | \ V\in \C^2 \right\}=\{(x,y) \in \R^2 |  \ x\geq 0, \ y\geq 0\}\subset \R^2,$$
 and thus coincide with its Levi cone.
\end{example}

\begin{example}\label{ex2pi}
 The Levi range of the strongly Levi nondegenerate quadric $M\subset \C^4$ of codimension $2$ defined by   
$$A_1=
\begin{pmatrix}
	1& 0 \\ 0 &-1 
\end{pmatrix} \ \ \ \ A_2=
\begin{pmatrix}
	0& 1 \\ 1 &0 
\end{pmatrix}$$
is equal to
$$\left\{\left(|V_1|^2-|V_2|^2,V_1\overline{V_2}+\overline{V_1}V_2\right) \ | \ V\in \C^2 \right\}=\R^2.$$
It follows that the Levi range coincides once more the  Levi cone.
\end{example}

\begin{example}\label{exconet} Let $M_t\subset \C^4$ be the strongly pseudoconvex quadric of codimension $2$ defined by   
$$A_1=
\begin{pmatrix}
	1& 0 \\ 0 &1 
\end{pmatrix} \ \ \ \ A_2=
\begin{pmatrix}
	1& 0 \\ 0 & t
\end{pmatrix}$$
with $0\leq t<1$.  Its Levi range, which also coincides with its Levi cone, is equal to the cone 
$$\left\{\left(|V_1|^2+|V_2|^2,|V_1|^2+t|V_2|^2\right) \ | \ V\in \C^2 \right\}=\{(x,y) \in \R^2 \ | \  0 \leq t x \leq  y \leq x\}.$$
\end{example}

\begin{example}
We  now consider the $\mathfrak{D}$-nondegenerate quadric $M\subset \C^5$ of codimension $2$ defined by   
$$A_1=
\begin{pmatrix}
	1& 0 & 0 \\ 0 &-1 & 0 \\ 0 &0 & 0
\end{pmatrix} \ \ \ \ A_2=
\begin{pmatrix}
	0& 0 & 0 \\ 0 &0 & 0 \\ 0 &0 & 1
\end{pmatrix}.$$
Its Levi range is 
 $$\left\{\left(|V_1|^2-|V_2|^2,|V_3|^2\right) \ | \ V\in \C^3 \right\}=\{(x,y) \in \R^2 |  \ y\geq 0\}\subset \R^2.$$
 and coincides once more with its Levi cone. Note that it contains the entire line (through the origin) $y=0$.  
\end{example}    

\begin{example}\label{exrank}
Let $M\subset \C^5$ be the strongly pseudoconvex quadric of codimension $3$ defined by   
$$A_1=
\begin{pmatrix}
	1& 0 \\ 0 &0 
\end{pmatrix} \ \ \ \ 
A_2=
\begin{pmatrix}
	0& 0 \\ 0 &1 
\end{pmatrix}\ \ \ \ 
A_3=
\begin{pmatrix}
	0& 1 \\ 1 &0 
\end{pmatrix}.$$
Note that $M$ is $\mathfrak{D}$-nondegenerate (e.g. one can choose $A=A_1+A_2$ and $V=\transp(i,1)$) but is not fully nondegenerate due to dimensional reasons.
Its Levi range
 $$\left\{\left(|V_1|^2,|V_2|^2,V_1\overline{V_2}+\overline{V_1}V_2\right) \ | \ V\in \C^2 \right\}$$ 
 is the convex region $\{(x,y,2a\sqrt{xy}) \ | \ x,y\geq 0 \ \  -1\leq a \leq 1\}$. 
\end{example}

\begin{example}\label{exline}
The quadric $M\subset \C^7$ of codimension $3$ defined by   
$$A_1=
\begin{pmatrix}
	1& 0 & 0 & 0 \\ 0 &0 & 0 & 0 \\ 0 &0 & 0 & 0 \\ 0 &0 & 0 & 0
\end{pmatrix} \ \ \ \ 
A_2=
\begin{pmatrix}
	0& 0 & 0 & 0 \\ 0 &1 & 0 & 0 \\ 0 &0 & 0 & 0 \\ 0 &0 & 0 & 0
\end{pmatrix}\ \ \ \ 
A_3=
\begin{pmatrix}
	0& 0 & 0 & 1 \\ 0 &0 & 1 & 0 \\ 0 &1 & 0 & 0 \\ 1 &0 & 0 & 0
\end{pmatrix}.$$
 satisfies Condition $\aaa$ and is strongly Levi nondegenerate. Its Levi range, also equal to its Levi cone, is 
  $$\left\{\left(|V_1|^2,|V_2|^2,2\Re e (V_1\overline{V_4}+V_2\overline{V_3})\right) \ | \ V\in \C^4 \right\} = \{ (x,y,t) \ | \ x, y\geq 0\} \setminus \left(\{0\}\times \{0\} \times \R^*\right).$$ 
It follows that its Levi cone
 does not contain any line through 0. Note however that it contains all the lines of the form $\{(a,b,t) \ | \ t \in \R\}$ with $a,b \geq 0$ not both zero.  Moreover  $M$ is not strongly pseudoconvex. Indeed, consider any real linear combination 
 $$A:=aA_1+bA_2+cA_3=\begin{pmatrix}
	a& 0 & 0 & c \\ 0 &b & c & 0 \\ 0 &c & 0 & 0 \\ c &0 & 0 & 0\\
\end{pmatrix}.$$
For $V=\transp (1,1,V_3,0)$ we have 
 $$\transp \overline{V} A V= a+b+2c\Re e V_3$$
 which can be  negative or positive e.g. by taking $\Re e V_3$ arbitrarily large (positively or negatively). Alternatively, one can see that $A$ has   
 eigenvalues of opposite signs, namely two positive ones $\frac{a+\sqrt{a^2+4c^2}}{2}$ and $\frac{b+\sqrt{b^2+4c^2}}{2}$,  and two negatives ones 
 $\frac{a-\sqrt{a^2+4c^2}}{2}$ and $\frac{b-\sqrt{b^2+4c^2}}{2}$.      
 
\end{example}

\section{The Levi range and the defect of analytic discs}\label{sec2}

In this section, we investigate the connection between the Levi range and the notion of $\mathfrak{D}$-nondegeneracy. More precisely, we are interested in characterizing the $\mathfrak{D}$-nondegeneracy by means of the Levi range in a spirit similar to Proposition \ref{propcone1}. We recall that if a real submanifold $M \subset \C^N$ of codimension $d$ of the form \eqref{eqred} is  $\mathfrak{D}$-nondegenerate at $0$, then in particular there exists $V \in \C^n$ such that $A_1V,\ldots,A_nV$ are $\R$-linearly independent. This property turns out to have a particular connection with the Levi range of $M$.    

\begin{theo}\label{theorange}
Let  $M$ be a  real submanifold of the form \eqref{eqred}. There exists $V \in \C^n$ such that $A_1V,\ldots,A_dV$ are $\R$-linearly independent if and only if the Levi range of $M$ at $0$ has nonempty interior. 
\end{theo}
\begin{proof}
Assume first that there exists $V \in \C^n$ such that the vectors $A_1V,\ldots,A_dV$ are $\R$-linearly independent, and consider the Levi map $L_0: \C^n \to \R^d$
$$L_0(W)=\left(\transp \overline{W}A_1W,\ldots,\transp \overline{W}A_dW\right).$$
Its differential at $V$, denoted by $d_VL_0: \C^n \to \R^d$, is given by 
\begin{eqnarray*}
d_VL_0(W)&=&\transp\left(\transp \overline{V}A_1W+\transp \overline{W}A_1V,\ldots,\transp \overline{V}A_dW+\transp \overline{W}A_dV\right)\\
\\
& =& 2\transp\left(\Re e (\transp \overline{W}A_1V),\ldots,\Re e (\transp \overline{W}A_dV)\right)\\
\\
& =& 2 \Re e \left(\transp D \overline{W}\right).\\
\end{eqnarray*}
Here we recall that $D$ is the $n \times d$ matrix whose $j^{th}$ column is $A_jV$.  
As a linear map from $\R^{2n}\simeq \C^n$ to $\R^d$, the map $d_VL_0$ is equal, after reordering its columns, to 
$$d_VL_0=(\Re e (\transp D) \ \ \Im m  (\transp D)).$$
By assumption, the rows of $\transp D$ are $\R$-linearly independent, implying that the real differential $d_VL_0: \R^{2n} \to \R^d$ is of rank $d$, and is thus surjective.
It follows that the range of the Levi map contains an open set centered at the point $L_0(V) \in \R^d$. 

Before proving the converse, note that we have proved that $A_1V,\ldots,A_dV$ are $\R$-linearly independent if and only if the real differential $d_VL_0$ has maximal real rank $d$.  Assume now that the Levi range of $M$ contains an open set $O \subset \R^d$. According to Sard's theorem, the set of values $L_0(V)\in \R^d$ such that the differential $d_VL_0:\R^{2n}\to \R^d$ has rank less than $d$ has zero measure. It follows that there is $V \in \C^n$ with $L_0(V) \in O$ such that the differential  $d_VL_0$ has maximal rank $d$, and thus, $A_1V,\ldots,A_dV$ are $\R$-linearly independent. 
\end{proof}

As a direct consequence, we have the following two important corollaries. 
\begin{cor}
Let  $M$ be a   real submanifold of the form \eqref{eqred}. If M is $\mathfrak{D}$-nondegenerate at $0$  then its Levi range at $0$ has nonempty interior. In case $M$ is strongly pseudoconvex these are equivalent statements.
\end{cor}

For dimensional reasons, we have
\begin{cor}
If $M$ is strongly Levi nondegenerate with $2n<d\leq n^2$ then  the interior of its Levi range at $0$ is empty.
\end{cor}

In light of \cite{be-me}, we can  connect the rank of the Levi map with the defect of analytic discs. We recall these notions now. We denote by $
\Delta$ the unit disc in $\C$ and by $\partial \Delta$ its boundary. An {\it analytic disc} is a smooth map $f: \overline{\Delta} \to \C^N$ which is holomorphic on $\Delta$. Such a disc is {\it attached to M} if $f(\partial \Delta) \subset M$. Assume now that $M$ is of the form \eqref{eqred} and denote by $r_1,\ldots,r_d$ its  defining functions (corresponding to each line of  \eqref{eqred}). 

Now, according to \cite{tu0} (see also \cite{ba-ro-tr}, \cite{ber}), an analytic disc $f$ attached to $M$ is {\it defective} if there exists a continuous 
map 
$c: \partial \Delta \to \R^{d}$ such that the map 
$$\zeta \mapsto \sum_{j=1}^d c_j(\zeta) \partial r_j(f(\zeta),\overline{f(\zeta)})$$
defined on $\partial \Delta$ extends holomorphically on $\Delta$. The disc is  {\it nondefective} if it is not defective. We recall that the existence of a small nondefective disc ensures the wedge extendability of CR maps  \cite{tu0} (see also Theorem 8.6.1 \cite{ber}).   
We also define the {\it defect} of an analytic  disc attached to $M$ by
$${\rm def}(f):={\rm dim}_\R \{c:\partial \Delta \to \R^d \ | \ \zeta \mapsto c(\zeta)\partial r(f(\zeta),\overline{f(\zeta)}) \mbox{ extends holomorphically on } \Delta\}.$$
Thus, $f$ is nondefective if and only if  ${\rm def}(f)=0$. 

It is important to note an analytic disc $f=(h,g)$ attached to a {\it model quadric} $M$ is defective whenever there exists a constant $c=(c_1,\ldots,c_d) \in \R^d\setminus\{0\}$ such that the map 
$$\zeta \mapsto  \sum_{j=1}^d c_j \partial_z r_j(h(\zeta),\overline{h(\zeta)})$$
defined on $\partial \Delta$ extends holomorphically on $\Delta$. This  follows from the fact that 
$\partial_w r_1,\ldots \partial_w r_d$ are constant. In that case, the defect of such a  disc is simply 
$${\rm def}(f)={\rm dim}_\R \{c \in \R^d \ | \ \zeta \mapsto  \sum_{j=1}^d c_j \partial_z r_j(h(\zeta),\overline{h(\zeta)}) \mbox{ extends holomorphically on } \Delta\}.$$

We then recall the following  lemma.
\begin{lemma}[\cite{be-me}]\label{lemclio}
Let $M \subset \C^N$ be a model quadric and let $f=(h,g)$ be an analytic disc attached to $M$ with $h$ of the form 
\begin{equation}\label{eqdisc}
h(\zeta)=(1-\zeta)V
\end{equation}
 for some $V \in \C^n$. Then $f$ is nondefective if and only if $A_1V,\ldots,A_dV$ are $\R$-linearly independent.
\end{lemma} 
We emphasize that analytic discs of the form \eqref{eqdisc} are particularly important. They are indeed {\it stationary} in the sense of \cite{le,tu} and  
play a key role in obtaining the $2$-jet determination of CR automorphisms of  smooth perturbations of $M$ \cite{be-bl-me,be-me,tu3}.  

In the context of the previous lemma, and still denoting by $D$ the $n\times d$ matrix whose $j^{th}$ column is $A_jV$, we then have the following straightforward result. 
\begin{cor}\label{cordisc}
Let  $M$ be a   model  quadric  and let $f=(h,g)$ be an analytic disc for $M$ with $h$ of the form \eqref{eqdisc}. Then  
$${\rm def}(f)=\dim {\rm ker}D_{|\R^d}=d-{\rm dim} \ {\rm span}_\R\{A_1V,\ldots,A_dV\}$$
and the rank of the real differential of Levi map $L_0$ at $V$ is  
$${\rm rank}(d_VL_0)=d - {\rm def}(f).$$ 

In particular, the Levi range at $0$ of $M$ contains an open set if and only if $M$ admits a nondefective analytic disc of the form \eqref{eqdisc}.   
\end{cor}

We now illustrate the results obtained in the present section so far with an example borrowed from \cite{gr-me}. 

\begin{example}
Consider in $\C^{10}$, the submanifold of codimension $4$ defined by  
$$A_1=
\begin{pmatrix}
	0& -i & 0 & 0 & 0 & 0 \\ i &0 & 0 & 0 & 0 & 0 \\ 0 &0 & 0 & 0 & 0 & 0 \\ 0 &0 & 0 & 0 & 0 & 0 \\ 0 &0 & 0 & 0 & 0 & 0 \\ 0 &0 & 0 & 0 & 0 & 0
\end{pmatrix} \ \ \ \ 
A_2=
\begin{pmatrix}
	0& 0 & 0 & 0 & 0 & 0 \\ 0 &0 & -i & 0 & 0 & 0 \\ 0 &i & 0 & 0 & 0 & 0 \\ 0 &0 & 0 & 0 & 0 & 0 \\ 0 &0 & 0 & 0 & 0 & 0 \\ 0 &0 & 0 & 0 & 0 & 0
\end{pmatrix}$$
$$A_3=
\begin{pmatrix}
	0& 0 & -i & 0 & 0 & 0 \\ 0 &0 & 0 & 0 & 0 & 0 \\ i &0 & 0 & 0 & 0 & 0 \\ 0 &0 & 0 & 0 & 0 & 0 \\ 0 &0 & 0 & 0 & 0 & 0 \\ 0 &0 & 0 & 0 & 0 & 0
\end{pmatrix} \ \ \ \ 
 A_4=
\begin{pmatrix}
	0& 0 & 0 & 0 & 0 & 1 \\ 0 &0 & 0 & 0 & 1 & 0 \\ 0 &0 & 0 & 1 & 0 & 0 \\ 0 &0 & 1 & 0 & 0 & 0 \\ 0 &1 & 0 & 0 & 0 & 0 \\ 1 &0 & 0 & 0 & 0 & 0
\end{pmatrix}.$$
Note first for $V_0=\transp(1,1,i,0,0,0)$, the vectors $A_1V_0,A_2V_0,A_3V_0,A_4V_0$ are $\R$-linearly independent. Thus, the interior of its Levi range is nonempty. The Levi map $L_0: \C^6 \to \R^4$ is equal to
\begin{eqnarray*}
L_0(V)& =& (2\Im m (\overline{V_1}V_2),2\Im m (\overline{V_2}V_3),2\Im m (\overline{V_1}V_3),2\Re e (V_1\overline{V_6}+V_2\overline{V_5}+V_3\overline{V_4})\\
\\
& =& 2(x_1y_2-y_1x_2,x_2y_3-y_2x_3,x_1y_3-y_1x_3,x_1x_6+y_1y_6+x_2x_5+y_2y_5+x_3x_4+y_3y_4)\\
\end{eqnarray*} 
where $V_j=x_j+iy_j, j=1,\ldots,4.$ It follows that its differential is given by 
$$d_VL_0=
2\begin{pmatrix}
	y_2& -x_2 & -y_1 & x_1 & 0& 0 & 0 & 0 & 0 & 0  & 0 & 0 \\ 
	0& 0  & y_3 & -x_3 & -y_2 & x_2 & 0 & 0 & 0& 0  & 0 & 0\\ 
	y_3& -x_3 & 0 & 0  & -y_1 & x_1 & 0 & 0 & 0  & 0 & 0 & 0\\ 
	x_6& y_6 & x_5 & y_5& x_4 & y_4 & x_3 & y_3 & x_2 & y_2 & x_1 & y_1 \\ 
\end{pmatrix}.$$
The rank of $d_VL_0$ may be equal to $4$, $3$ (e.g for $V_1=V_2=0$ and $V_3 \neq 0$), $1$ (e.g. for $V_1=V_2=V_3=0$ and  $V_6 \neq 0$), and $0$ (for $V=0$). Note that it cannot however be equal to $2$ essentially due to the last row. Rephrasing now in terms of defect by considering an analytic disc $f=(h,g)$ for $M$. If 
$h(\zeta)=(1-\zeta,1-\zeta,i(1-\zeta),0,0,0)$ then $f$ is nondefective. 
When $h(\zeta)=(0,0,1-\zeta,0,0,0)$ then the defect of $f$ is equal to 1, while it is 3 when  $h(\zeta)=(0,0,0,0,0,1-\zeta)$. No analytic disc with 
$h$ of the form $h(\zeta)=(1-\zeta)V$ has defect 2. Finally, note that although the Levi map reaches the maximal rank $4$, $CR$ automorphisms of $M$ are determined by their $3$-jet but not their $2$-jet.          
\end{example}

\begin{remark}
The previous example also shows that not all ranks are achievable. The consequences of such gaps are relatively clear on the geometry of the range in $\R^d$ but do not seem to have much impact on mapping problems for $M$. More precisely, the Levi map of the submanifold
defined by 
$$
\begin{pmatrix}
	1& 0 & 0 \\ 0 &0 & 0 \\ 0 &0 & 0
\end{pmatrix} \ \ \ \ \begin{pmatrix}
	0& 0 & 0 \\ 0 &1 & 0 \\ 0 &0 & 0 
\end{pmatrix}\ \ \ \ \begin{pmatrix}
	0& 0 & 0 \\ 0 &0 & 0 \\ 0 &0 & 1 
\end{pmatrix}$$
assumes all ranks in $\{0,1,2,3\}$ while the one of  the submanifold
considered in example \ref{exrank} assumes the ranks $0,2$, and $3$. Both submanifolds are strongly pseudoconvex submanifold and thus satisfy the 2-jet determination of CR-automorphisms \cite{tu3}.
\end{remark}

Assume that  $M$ given by \eqref{eqred} admits a vector $V \in \C^n$ such that $A_1V,\ldots,A_dV$ are $\C$-linearly independent; for instance, one can simply assume that $M$ is fully nondegenerate. We already know from Theorem \ref{theorange} that its Levi range has nonempty interior. This raises the following question. 
 
\vspace{0.1cm}

\noindent {\bf Question:} Can  something more be said? In other words, is there a property of the Levi map that distinguishes $\R$-linearity and  $\C$-linearity  of the family $A_1V,\ldots,A_dV$? 

\vspace{0.5cm}

Theorem \ref{theorange} fully characterizes submanifolds for which the Levi map assumes the maximal rank $d$. It turns out that when $d=2$ or $3$, this is always the case, under the necessary condition  that $n \ge 2.$

\begin{prop}\label{propd=2}
Let $M$ be a CR submanifold given by \eqref{eqred} of codimension $2$ in $\C^{n+2}$ with $n\ge 2$. If $M$ satisfies ($\mathfrak{a}$)  then its Levi range at $0$ has nonempty interior. 
\end{prop}

We note that while this proposition was known and proved by the authors, we present a  proof whose first part is due to Tumanov from a private communication.  
\begin{proof}
Let $A_1,A_2$ be two $n\times n$  matrices which are linearly independent. 

We first assume  that one of these matrices, say $A_1$, is invertible.
Assume by contradiction that for all  $V \in \C^n$, there exist $\lambda$  and  $\mu$  not all zero such that $\mu A_2V=\lambda A_1V.$ Note that since $A_1$ is invertible, $\mu \ne 0$ for all $V\ne 0.$  Without loss of generality, we  then assume that there exists  $\lambda$  such  that  $A_1^{-1}A_2V=\lambda V$. Since any vector is an eigenvector of  $A_1^{-1}A_2$, it follows that $A_1^{-1}A_2$ is diagonalizable, and hence all the eigenvalues are equal. The matrix $A_1^{-1}A_2$ is then  a scalar multiple of the identity, contradicting the fact that $A_1$ and $A_2$ are linearly independent.

\vspace{0.2cm} 
Assume now that none of the matrices $A_1,A_2$ is invertible. After a change of coordinates, we suppose that $A_1$ is of the form 
$$A_1=
 \begin{pmatrix} \tilde{A_1} & 0 \\ 
 0 & 0 \\
 \end{pmatrix} \ \ \ A_2= 
 \begin{pmatrix} \tilde{A_2} & \transp \overline{X} \\ 
 X & Y \\
 \end{pmatrix} \ \ \ $$
where $ \tilde{A_1}$ is a $n_1\times n_1$ diagonal matrix with $\pm 1$ and, setting $n_2=n-n_1$  where  $\tilde{A_2}, X$ and $Y$ are matrices of respective size $n_1\times n_1$,  $n_2 \times n_1$ and $n_2 \times n_2$. We distinguish three cases.    

\noindent \underline {Case 1:} if $X \neq 0$  we then take $V=\transp (U,0)$ with $U \in \C^{n_1} \setminus \{0\}$ and such that  $U \notin Ker X$. In which case $A_1V$ and $A_2 V$ are $\R$-linearly independent.

\noindent \underline {Case 2:} if $X=0$ and $Y \neq 0$, then taking $V=\transp (U,W)$ with $U \in \C^{n_1} \setminus \{0\}$ and $W \notin Ker Y$ implies that  $A_1V$ and $A_2 V$ are $\R$-linearly independent.

\noindent \underline {Case 3:} if both  $X$ and $Y$ are zero then the matrices $ \tilde{A_1}$ and $ \tilde{A_2}$ are linearly independent and the first part of the proof provides a vector $U$ such that  $\tilde{A_1}U$ and $ \tilde{A_2}U $ are $\R$-linearly independent. It follows that for $V=\transp (U,0)$, the vectors $A_1V$ and $A_2V$ are $\R$-linearly independent. 
\end{proof}

For codimension 3, we have  

\begin{theo}\label{theod=3}
Let $M$ be a CR submanifold of the form \eqref{eqred} of codimension $3$ in $\C^{n+3}$ with $n\geq 2$. If $M$ satisfies ($\mathfrak{a}$) and is strongly Levi nondegenerate then its Levi range at $0$ has nonempty interior. 
\end{theo}

In the proof of this theorem, we first consider the dimension  $n=2$ since the proof in that case is simpler and interesting on its own. The second part of the proof works in fact for any $n \geq 2$. Also, recall that $D$ is the $n\times d$ matrix whose $j^{th}$ column is $A_jV$.
    
\begin{proof} 
We start with the case  $n=2,$ that is, when the submanifold  $M$ is of codimension 3 in $\C^5$. The submanifold $M$ is given by 
$$A_j=\begin{pmatrix}a_j & c_j \\ \overline{c_j} & b_j \end{pmatrix}$$
with $a_j,c_j \in \R$ and $c_j=\alpha_j+i\beta_j \in \C$, $j=1,2,3$. Since $A_1,A_2,A_3$ are linearly independent, one of the four $3\times 3$ minors of the following matrix is nonzero
$$\begin{pmatrix}
a_1 & a_2 & a_3 \\ 
\alpha_1 & \alpha_2 & \alpha_3 \\ 
\beta_1 & \beta_2 & \beta_3 \\ 
c_1 & c_2 & c_3 \\ 
 \end{pmatrix}.$$
We denote by $m_\ell$ the $3\times 3$ minor of this matrix obtained by removing the $\ell^{th}$ row. We then consider four cases. 

\noindent \underline{Case 1:} if $m_4\neq 0$, we claim that for $V=\transp (1 \  0)$ the three vectors $A_1V,A_2V$ and $A_3V$ are $\R$-linearly independent. This follows from the fact that the 
minor $m_4$ is one of the $3\times 3$ minor in the matrix 
$$(\Re e (\transp D) \ \ \Im m  (\transp D))=\begin{pmatrix} 
a_1 & \alpha_1 & 0    & -\beta_1 \\ 
a_2 & \alpha_2 & 0   &-\beta_2  \\ 
a_3 & \alpha_3 & 0    & -\beta_3\\ 

\end{pmatrix}.$$  
  
\noindent \underline{Case 2:} if $m_1\neq 0$, then for $V=\transp (0 \  1)$ the vectors $A_1V,A_2V$ and $A_3V$ are $\R$-linearly independent since $m_1$ occurs in  
$$(\Re e (\transp D) \ \ \Im m  (\transp D))=
\begin{pmatrix} 
\alpha_1 & c_1  & \beta_1 & 0 \\ 
\alpha_2  & c_2 & \beta_2  & 0 \\ 
 \alpha_3  & c_3 & \beta_3& 0 \\ 

\end{pmatrix}.$$ 

\noindent \underline{Case 3:} now if $m_4=m_1=0$ and $m_3\neq 0$ then the vectors $A_1V,A_2V$ and $A_3V$, for $V=\transp (i \  1)$ are $\R$-linearly independent.  Indeed, the $3\times 3$ minor obtained in 
 $$(\Re e (\transp D) \ \ \Im m  (\transp D))=
\begin{pmatrix} 
\alpha_1 & c_1+\beta_1  & a_1+\beta_1 & \alpha_1 \\ 
\alpha_2  & c_2+\beta_2& a_2+\beta_2   & \alpha_2 \\ 
 \alpha_3  &c_3+\beta_3 & a_3+\beta_3 & \alpha_3 \\ 
\end{pmatrix} $$ 
 by removing its last column is exactly $m_3$.  

\noindent \underline{Case 4:} finally if $m_4=m_1=0$ and $m_2\neq 0$ then the vectors $A_1V,A_2V$ and $A_3V$, for $V=\transp (1 \  1)$ are $\R$-linearly independent since  the $3\times 3$ minor obtained in 
 $$(\Re e (\transp D) \ \ \Im m  (\transp D))=
\begin{pmatrix} 
a_1+\alpha_1 &  c_1+ \alpha_1 & \beta_1  &  -\beta_1 \\ 
a_2+\alpha_2 &  c_2+ \alpha_2  & \beta_2& -\beta_2 \\ 
a_3+\alpha_3 &   c_3+ \alpha_3  &\beta_3 & -\beta_3 \\ 
\end{pmatrix} $$ 
by removing its last column is exactly $m_2$. This proves the theorem in case $n=2$.    
         
\vspace{0.5cm}

We now suppose $n>2$. The submanifold $M$ is defined by three linearly independent Hermitian $n\times n$ matrices $A_1,A_2,A_3$. 
After a linear holomorphic change, we may assume that $A_1$ is  invertible.  We also suppose by 
contradiction that for all $V\in \C^n$, the three vectors $A_1V,A_2V,A_3V$ are $\R$-linearly dependent, that is, there are three real numbers $\lambda_1(V),\lambda_2(V),\lambda_3(V) \in \R$, not all zero, such that 
\begin{equation}\label{tutu}\lambda_1(V)A_1V+\lambda_2(V)A_2V+\lambda_3(V)A_3V=0.
\end{equation}
We define for any real number $k\neq 0$, 
$$E_c:=\{ V\in \C^n \ | \ \lambda_2(V) =k \lambda_3(V)\}.$$  
For $V  \in E_c \setminus\{0\},$ we have  $\lambda_2(V) =c  \lambda_3(V) \ne 0,$ and hence we may rewrite $E_c$ as 
$$E_c=\{ V\in \C^n \ | \ \lambda(V) A_1V=A_2V+cA_3V, \ \lambda(V) \in \R \}.$$
Note that since  $A_1$ is invertible, any vectors in $E_c$ is an eigenvector of the matrix $A_1^{-1}(A_2+cA_3)$.   
We claim that  there exists $c \in \R\setminus \{0\}$ for which $ E_c$  is  a real  algebraic variety  near some point of $E_k$ of dimension $2n-1$.  

Since $A_2$ and $A_3$ are linearly indepdendent, by   Proposition \ref{propd=2}, there exists $V \in \C^n$ such that $A_2V,A_3V$ are $\R$-linearly independent. Note in that case that $\lambda_1(V) \neq 0$. So we may rewrite \eqref{tutu}, outside of a thin set $X \subset \C^n$, as  
 $$A_1V=\lambda_2(V)A_2V+\lambda_3(V)A_3V.$$
 Now denoting  by $c_j^\ell(V)$, $j=1,2,3$ and $\ell=1,\ldots,2n$ the coefficients of $A_jV \in \R^{2n}$, we write the above equation as 
$$ \begin{pmatrix} 
c_1^1(V) \\ 
\vdots \\ 
c_1^{2n}(V) \\ 
\end{pmatrix} = \begin{pmatrix} 
c_2^1(V) & c_3^1(V)  \\ 
\vdots & \vdots\\ 
c_2^{2n}(V) & c_3^{2n}(V)  \\ 
\end{pmatrix} \begin{pmatrix} 
\lambda_2(V) \\ 
\lambda_3(V) \\ 
\end{pmatrix},$$
where the matrix  $\begin{pmatrix} 
c_2^1(V) & c_3^1(V)  \\ 
\vdots & \vdots\\ 
c_2^{2n}(V) & c_3^{2n}(V)  \\ 
\end{pmatrix}$ is of rank $2.$
Accordingly, we consider two its linearly independent rows, say the 
$\ell_1^{th}$ and $\ell_2^{th}$ rows. It follows from Cramer's rule that, outside of $X$,
$$\lambda_2(V)=\frac{{\rm det}\begin{pmatrix} 
c_1^{\ell_1}(V) & c_2^{\ell_1}(V)\\
c_1^{\ell_2}(V) & c_2^{\ell_2}(V)\\ 
\end{pmatrix}}{{\rm det}\begin{pmatrix} 
c_2^{\ell_1}(V) & c_3^{\ell_1}(V)\\
c_2^{\ell_2}(V) & c_3^{\ell_2}(V)\\ 
\end{pmatrix}} \ \ \ \ \lambda_3(V)=\frac{{\rm det}\begin{pmatrix} 
c_3^{\ell_1}(V) & c_1^{\ell_1}(V)\\
c_3^{\ell_2}(V) & c_1^{\ell_2}(V)\\ 
\end{pmatrix}}{{\rm det}\begin{pmatrix} 
c_2^{\ell_1}(V) & c_3^{\ell_1}(V)\\
c_2^{\ell_2}(V) & c_3^{\ell_2}(V)\\ 
\end{pmatrix}}.$$

Choose   $V_0 \in \C^n \setminus X$ for which    $\lambda_2(V_0)\lambda_3(V_0) \neq 0$ and  set $c:=\lambda_2(V_0)/\lambda_3(V_0)$. 
The equation $\lambda_2(V)= c\lambda_3(V)$ can be written as  
\begin{equation*}
{\rm det}\begin{pmatrix} 
c_1^{\ell_1}(V) & c_2^{\ell_1}(V)\\
c_1^{\ell_2}(V) & c_2^{\ell_2}(V)\\ 
\end{pmatrix}-c{\rm det}\begin{pmatrix} 
c_3^{\ell_1}(V) & c_1^{\ell_1}(V)\\
c_3^{\ell_2}(V) & c_1^{\ell_2}(V)\\ 
\end{pmatrix}=0.
\end{equation*}
Since each of the involved coefficient is linear in $V$, the left handside of this equation is a homogenenous polynomial of degree 2 in $V\in \R^{2n}$, denoted by $P$. We then have  $\dim E_c=2n-1$ near $V_0.$  Recall that each vector  $V\in P^{-1}(\{0\}) $ is an eigenvector of $A^{-1}
(A_2+cA_3)$ with the corresponding eigenvalue $\lambda(V)$. Note that $\lambda(V)$ depends continuously on $V$ and is thus 
locally constant since $A_{1}^{-1}(A_2+cA_3)$ has a finite number of eigenvalues. We  assume that  $=\lambda(V)$ is constant in some
 small enough  neighborhood $U \subset P^{-1}(\{0\}$ of $V_0$ and we denote this constant by $\lambda_0$.   We now distinguish two cases. 

\noindent \underline{Case 1:} if $U$ is not included in a real hyperplane of $\R^{2n}$, we can then find a basis of $\R^{2n}$ consisting of 
$2n$ linearly independent vectors in $U$, which implies that 
$$A_1^{-1}(A_2+cA_3)=\lambda_0 I_n$$ is a scalar multiple of the identity contradicting the fact that $A_1,A_2,A_3$ are linearly independent. 

\noindent \underline{Case 2:} if $U$ lies in a real hyperplane $\R^{2n}$,  we then have $\lambda_0 A_1V=(A_2 + cA_3)V$ on an open set of $\Bbb C^n$ by the unicity property for CR maps. This also contradicts the fact that $A_1,A_2,A_3$ are linearly independent.
\end{proof}

We now provide two examples, of codimensions 4 and 5,  that illustrate that beyond codimension 3, the Levi map does not reach full rank in general. 

\begin{example}
We consider in $\C^6$, the strongly pseudoconvex quadric of codimension $4$ defined by  
$$A_1=
\begin{pmatrix}
	1& 0  \\ 0 &0  
\end{pmatrix} \ \ \ \ 
A_2=
\begin{pmatrix}
	0 & 0  \\  0 & 1 \\ 
\end{pmatrix}\ \ \ \ 
A_3=
\begin{pmatrix}
	0& 1  \\  1 & 0 \\
\end{pmatrix}\ \ \ \  
A_4=
\begin{pmatrix}
	0& -i  \\ i  & 0 \\
\end{pmatrix}.$$
The Levi map $L_0: \C^2 \to \R^4$ is 
$$L_0(V)=(|V_1|^2,|V_2|^2,2\Re e (\overline{V_1}V_2),2\Im m (\overline{V_1}V_2)).$$ 
Writing $V_j=x_j+iy_j$, we have  
$$L_0(V)=(x_1^2+y_1^2,x_2^2+y_2^2,2x_1x_2+2y_1y_2,2x_1y_2-2x_2y_1).$$  
Its real differential at $V$ is given by 
$$d_VL_0=2
\begin{pmatrix}
x_1& y_1 & 0 & 0     \\ 
0& 0  & x_2 & y_2   \\ 
x_2& y_2 &x_1 & y_1    \\ 
y_2& -x_2 & -y_1 & x_1 \\ 
\end{pmatrix}$$
and is never of rank $4$. In fact the rank of the Levi map is $3$ whenever $V\neq 0$, and is $0$ when $V=0$. It follows that for all $V\in \C^2$ the vectors $A_1V,\ldots,A_4V$ are always $\R$-linearly dependent. It is interesting to note that according to Tumanov \cite{tu3}, this submanifold admits nondefective analytic  discs; such discs are not of the form $f=(h,g)$ with $h(\zeta)=(1-\zeta)V$. Finally, as pointed out by Boggess (see p. 205 \cite{bo}), note that the Levi range is not convex.  
\end{example}

The next example is taken from \cite{me} (see also \cite{gr-me}). 

\begin{example}\label{exfer}
Consider now in  $\C^9$, the quadric $M$ of codimension $5$ defined by  
$$A_1=
\begin{pmatrix}
	0& 1 & 0 & 0 \\ 1 &0 & 0 & 0 \\ 0 &0 & 0 & 0 \\ 0 &0 & 0 & 0
\end{pmatrix} \ \ \ \ 
A_2=
\begin{pmatrix}
	0 & -i & 0 & 0 \\ i  & 0 & 0 & 0 \\ 0 &0 & 0 & 0 \\ 0 &0 & 0 & 0
\end{pmatrix}\ \ \ \ 
A_3=
\begin{pmatrix}
	0& 0 & 0 & 1 \\ 0 &0 & 1 & 0 \\ 0 &1 & 0 & 0 \\ 1 &0 & 0 & 0
\end{pmatrix}$$
$$ A_4=
\begin{pmatrix}
	1& 0 & 0 & 0 \\ 0 &0 & 0 & 0 \\ 0 &0 & 0 & 0 \\ 0 &0 & 0 & 0
\end{pmatrix} \ \ \ \ 
A_5=
\begin{pmatrix}
	0& 0 & 0 & 0 \\ 0 & 1 & 0 & 0 \\ 0 &0 & 0 & 0 \\ 0 &0 & 0 & 0
\end{pmatrix}.$$
$M$ is  strongly Levi nondegenerate, satisfies $\aaa$ and is not $\mathfrak{d}$-nondegenerate. Indeed, for all $V \in \C^4$, the vectors $A_1V,\ldots,A_5V$ are $\R$-linearly dependent. This can be checked directly from the vanishing of all $5 \times 5$ minors of the below matrix $d_VL_0$. Moreover, we know from \cite{me} that CR automorphisms of $M$ are not determined by their $2$-jet, while CR automorphisms of $\mathfrak{d}$-nondegenerate sumbanifolds are determined by their $2$-jet \cite{be-me}. The Levi map $L_0: \C^4 \to \R^5$
$$L_0(V)=(2 \Re e(\overline{V_1}V_2),2\Im m (\overline{V_1}V_2),2\Re e (\overline{V_1}V_4+\overline{V_2}V_3),|V_1|^2,|V_2|^2).$$  
Writing $V_j=x_j+iy_j$, we have  
$$L_0(V)=(2x_1x_2+2y_1y_2,2x_1y_2-2y_1x_2,2x_1x_4+2x_2x_3+2y_1y_4+2y_2y_3,x_1^2+y_1^2,x_2^2+y_2^2).$$  
The Levi map $L_0$ is nowehere a submersion. Indeed, its differential is given by 
$$d_VL_0=2
\begin{pmatrix}
	x_2& y_2 & x_1 & y_1  & 0& 0& 0 & 0 \\ 
	y_2& -x_2 & -y_1 & x_1  & 0& 0& 0 & 0 \\ 
	x_4& y_4 & x_3 & y_3  &x_2& y_2& x_1 & y_1 \\ 
	x_1& y_1 & 0 & 0  & 0& 0& 0 & 0 \\ 
	0& 0 & x_2 & y_2  & 0& 0& 0 & 0 \\ 

\end{pmatrix}$$
and is of rank less than $5$ for all $V \in \C^4$. It is in fact of rank $4$  whenever $V_1\neq 0$ or $V_2 \neq 0$  since one can find a nonzero $4\times 4$ minor. If $V_1=V_2=0$, then its rank is $1$ if $V_3\neq 0$ or $V_4\neq 0$, and is equal to $0$ at $V=0$. \end{example}

\noindent {\bf Question:} In relation with the submanifold $M$ of Example \ref{exfer}, while all  analytic discs $f=(h,g)$ for $M$ with $h(\zeta)=(1-\zeta)V$ are defective, 
it is unclear whether or not $M$ admits a nondefective stationary disc.

\section{The  Levi range and the Levi cone}\label{sec3}
Let $M$ be a  CR sumbanifold of the form \eqref{eqred} with $d$  linearly independent $n\times n$  Hermitian matrices  $A_1,\ldots,A_d$. We know from Theorem \ref{theorange} 
that in case there is no $V\in \C^n$ such that $A_1V,\ldots, A_dV$ are $\R$-linearly independent then the Levi map 
does reach the maximal rank $d$, and thus, the Levi range differs with the Levi cone.  We then assume that there exists $V\in \C^n$ such that $A_1V,\ldots, A_dV$ are $\R$-linearly independent and we aim to understand when does the Levi range coincide with the Levi cone. 
In general, it is natural to understand geometric conditions that ensure that the Levi range coincides with the Levi range, and its consequences on mapping classification problems.  
 We need the following definition and lemma.

\begin{defi}
Let $A$ be a $n\times n$ Hermitian matrix. The {\it light cone of $A$} is the set
$$C_A=\{V\in \C^n \ | \ \transp \overline{V}AV=0\}.$$      
\end{defi}

Note that the intersection of the respective light cones of the matrices  $A_1,\ldots,A_d$ corresponds to the zero set of the Levi map
$$L_0^{-1}(0)=\{V\in \C^n \ | \ \transp \overline{V}A_1V=\ldots=\transp \overline{V}A_dV=0\}=C_{A_1} \cap C_{A_2} \cap \ldots \cap C_{A_d}.$$

The following definition is convenient for our purpose. 
\begin{defi}\label{defweak}
The submanifold $M$ is {\it weakly pseudoconvex} if there exist $c_1,\ldots,c_d\in \R$ such that the matrix $\sum_{j=1}^dc_jA_j$ is nonnegative, that is,  for all $V \in \C^n$ we have $\sum_{j=1}^dc_j\transp\overline{V}A_jV \geq 0$.   
\end{defi}
\begin{remark}\label{remlight}
In case $M$ is weakly pseudoconvex, its  Levi range at $0$, and thus the Levi cone at $0$, both lie on one side of a real hyperplane, e.g. $\sum_{j=1}c_jx_j=0$ where the $c$'s are the ones of Definition \ref{defweak}.  Conversely, if the Levi range is contained on one side a real hyperplane in $\R^d$ then  $M$ is  weakly pseudoconvex. In particular, when the codimension $d$ is equal to $2$, $M$ is  weakly pseudoconvex if and only if its Levi range lies in a half plane.   
\end{remark}

We now state the main result of this section. 

\begin{theo}\label{theorangecone}
Let $M \subset \C^{n+2}$ be a submanifold of codimension 2 of the form \eqref{eqred} with  two linearly independent $n\times n$ matrices $A_1,A_2$. Then
\begin{enumerate}[i.]
\item If $M$ is weakly pseudoconvex then the Levi range coincides with the Levi cone.
\item In case $n=2$, the Levi range  coincides with the Levi cone.
\end{enumerate}
\end{theo}

\begin{proof}
We prove the first point $i.$ Since $M$ is weakly pseudoconvex, we can assume that after a change of coordinates, $A_1$ is of the form  
$$A_1=\begin{pmatrix}
I_{n_1} & 0 \\
0 & 0\\
\end{pmatrix} $$
where $I_{n_1}$ denotes the identity matrix of size $n_1$. In that case, the zero set of the Levi map is
$$L_0^{-1}(0)=C_{A_1} \cap C_{A_2} \subset C_{A_1}={\rm Ker } A_1.$$
Hence $L_0^{-1}(0)$ is either $\{0\}$ (when $M$ is strongly pseudoconvex) or contained in a subspace of $\C^n$ of complex dimension $n-n_1$. It follows that  $\C^n \setminus L_0^{-1}(0)$ is connected and since 
 $L_0(\C^n\setminus L_0^{-1}(0))=L_0(\C^n) \setminus \{0\}$, then the Levi range minus $\{0\}$ is connected as well. Moreover, due to the form of $A_1$, the Levi range is contained in the half plane $\{ (x,y) \in \R^2 \ | \ x \geq 0\}$.     
 Finally, recall that if $p \in L_0(\C^n)$ then the semi line through $p$ is contained in $L_0(\C^n)$. It follows that the Levi range is a  cone with (vertex) angle $0<\theta<\pi$ in a half plane and thus coincides with the Levi cone.

\vspace{0.5cm}
We now move to  $ii$. We do not use the classification given Boggess (see \cite{bo}). Instead we present a proof based on elementary techniques. In view of $i.$, we consider  $M \subset \C^4$ which is not weakly pseudoconvex case. The matrices $A_1, 
A_2$ are linearly independent and such that for any $c_1,c_2\in \R$, there exist $V,W \in \C^2$ such that 
$$c_1\transp\overline{V}A_1V+ c_2\transp\overline{V}A_2V > 0 \ \ \mbox{ and } \ \ c_1\transp\overline{W}A_1W+ 
c_2\transp\overline{W}A_2W < 0.$$
In other words, for any lines $c_1x+c_2y=0$ in $\R^2$ there exist points in the range lying on each of its side. We aim to show that 
the Levi range is equal to $\R^2$ and thus coincides with the Levi 
cone. 

First, note that each of $A_1,A_2$ has two eigenvalues of opposite sides. Furthermore, we assume that 
$$A_1=\begin{pmatrix}
1 & 0 \\
0 & -1\\
\end{pmatrix} \ \mbox{ and } A_2=\begin{pmatrix}
a & c \\
\overline{c} & b\\
\end{pmatrix}\ 
.$$
 It is important to note that $c\neq0$; otherwise since $A_1$ and $A_2$ are both diagonal with opposite sign terms and linearly independent one may find a positive real linear combination of $A_1$ and $A_2$. 
 
For all $c_1,c_2 \in \R$, not all zero,  the matrix $c_1A_1+c_2A_2$ has eigenvalues of opposite signs. This occurs exactly when 
\begin{equation}\label{eqdet}
\det (c_1A_1+c_2A_2)=-c_1^2+(b-a)c_1c_2+(ab-|c|^2)c_2^2 < 0.
\end{equation}
Note that $ab-|c^2|<0$ since $A_2$ has itself eigenvalues of opposite signs. The determinant in $\eqref{eqdet}$ is negative whenever the discriminant
$$(b-a)^2+4(ab-|c|^2)=(a+b)^2-4|c|^2$$
is itself negative. We then assume from now on that 
\begin{equation}\label{eqa+b}
(a+b)^2<4|c|^2.
\end{equation}
We now show that the line $x=0$ is included in the range. 
The Levi map is equal to $$L_0(V)=\left(|V_1|^2-|V_2|^2,a|V_1|^2+b|V_2|^2+2\Re e (c\overline{V_1}V_2)\right).$$
Choose first $V \ \in \C^2$ in such a way that $V_2=V_1e^{i\phi}$. The second component of the Levi map is then 
$$a|V_1|^2+b|V_2|^2+2\Re e (c\overline{V_1}V_2)=|V_1|^2\left(a+b+2\Re e (ce^{i\phi})\right)=|V_1|^2\left(a+b+2|c|\cos (\tilde{\phi})\right)$$
for some $\tilde{\phi}$. 
In view of Equation \eqref{eqa+b}, the last term $a+b+2|c|\cos (\tilde{\phi})$ may change sign. Recall that if a point $L_0(V)$ is in the Levi range, then the semi line through $0$ and $L_0(V)$  is contained in the range as well. This shows that   $x=0$ is contained in the range.   
 We now show that the range contains any other vertical lines. Let $\alpha>0$ and consider the line $x=\alpha$.  For $V \in \C^2$, we take $V_2=r \in \R$. We now choose $V$ in such a way that $V_1=\sqrt{\alpha+r^2}e^{i\phi}$. This ensures that $|V_1|^2-|V_2|^2=\alpha$.  
 The second component of the Levi map reads as
 $$a(\alpha+r^2)+br^2+2r\sqrt{\alpha+r^2}|c|\cos(\tilde{\phi})$$
 for some  $\tilde{\phi}$. In fact, we choose the angle $\phi$ to make sure that $\cos(\tilde{\phi})=1$. The second component of the Levi map is then 
 $$a\alpha+(a+b)r^2+2r\sqrt{\alpha+r^2}|c|.$$
The function  defined on $\R$ $r \mapsto a\alpha+(a+b)r^2+2r\sqrt{\alpha+r^2}|c|$ is onto. Indeed, it is continuous with limits $-\infty$ as $r \to -\infty$ and $\infty$ as $r\to \infty$ due to  $\eqref{eqa+b}$.

The treatment of the  line $x=\alpha$ with  $\alpha<0$ is very similar.  The second component of the Levi map is still written as 
 $$a\alpha+(a+b)r^2+2r\sqrt{\alpha+r^2}|c|.$$
 The map $r \mapsto a\alpha+(a+b)r^2+2r\sqrt{\alpha+r^2}|c|$ which is now defined on $(-\infty,-\sqrt{-\alpha}] \cup [\sqrt{-\alpha},\infty)$, is onto as well. Indeed, the map is continuous on each of the two intervals and, using $\eqref{eqa+b}$ once more, has limits $-\infty$ as $r\to -\infty$ and  $\infty$ as $r\to \infty$; moreover, the sided limits as $r \to -\sqrt{-\alpha}^-$ and $\sqrt{-\alpha}^+$ coincide.  
 \end{proof}

The proof of $i.$ of Theorem \ref{theorangecone} works only for submanifolds of codimension 2. Indeed, although the Levi range (minus $\{0\}$) of a weakly pseudoconvex model $M$ of codimension $d>2$ remains connected, it may not be convex. This leads to the following question. 

\vspace{0.1cm}
\noindent {\bf Question:} it would then be interesting to find an explicit example of a weakly pseudoconvex quadric $M \subset \C^{n+d}$ with a nonconvex Levi range whose Levi map reaches maximal rank $d$. 
 \vspace{0.5cm}

As for the second point $ii.$ of Theorem \ref{theorangecone}, we believe the following question has a positiver answer. 

\vspace{0.1cm}
\noindent {\bf Question:} Let $M \subset \C^{n+2}$ be a submanifold of codimension 2 of the form \eqref{eqred} with  two linearly independent $n\times n$ matrices $A_1,A_2$. Do its Levi range and its Levi cone coincide?

 \vspace{0.5cm}

\begin{remark}
Note that in the proof of  each of the points in Theorem \ref{theorangecone}, we performed a linear change of coordinates. 
When $M$ is weakly pseudoconvex, the Levi range is a cone whose angle cannot exceed $\pi$, otherwise there exists a line with points on each of its side. When $M\subset \C^4$ is not  weakly pseudoconvex, the surjectivity of Levi map is invariant.  
 \end{remark}

In view of this remark and rephrasing Theorem \ref{theorangecone}, the Levi range of a submanifold of codimension $2$ is always a cone with vertex angle $\theta \in [0,2\pi].$ In case $n=2$, we refine the range of the angle to lie in $[0,\pi]$ or be equal to $2\pi$. In other words, nonconvex cones with angle $\theta \in (\pi,2\pi)$ cannot be the Levi range of a  quadric of codimension $2$ in $\C^4$. Moreover, in light of Example \ref{ex2pi} and Example \ref{exconet}, we can show that any cone of angle 
$\theta \in [0,\pi] \cup \{2\pi\}$ is the Levi range (or cone) of a quadric in $\C^4$ of codimension 2 (without referring to a linear change of coordinates).  For instance, the cone in the first quadrant bounded by the lines $y=ax$ and $y=bx$ with $b<a$ is the Levi range of the quadric  defined by 
  $$A_1=\begin{pmatrix}
\frac{1}{b}& 0 \\
0 & \frac{1}{a}\\
\end{pmatrix} \ \mbox{ and } A_2=\begin{pmatrix}
1 & 0 \\
0 & 1\\
\end{pmatrix}.$$
This leads to 

\vspace{0.1cm}
\noindent {\bf Question:} 
Is this a feature specific to  codimension 2 only? It would  be interesting for instance to focus on codimension $3$ and 
determine all possible Levi cones and whether any such region may be the Levi cone of a quadric of codimension 3.

\vskip 1cm

{\small
\noindent Florian Bertrand\\
Department of Mathematics\\\
American University of Beirut, Beirut, Lebanon\\{\sl E-mail address}: fb31@aub.edu.lb\\

\noindent Francine Meylan \\
Department of Mathematics\\
University of Fribourg, CH 1700 Perolles, Fribourg\\
{\sl E-mail address}: francine.meylan@unifr.ch\\
}

\end{document}